\documentclass[12pt]{amsart}
\usepackage{amsmath,amssymb,graphicx,latexsym,a4,longtable,nicefrac}
\usepackage[english]{babel}
\usepackage{color}
\def\bbbr{{\mathbb R}}
\def\bbbc{{\mathbb C}}
\def\bbbq{{\mathbb Q}}

\def\bbbp{{\mathbb P}}
\def\bbbz{{\mathbb Z}}

\def\cal#1{{\mathcal #1}}
\def\hat{\widehat}
\numberwithin{equation}{section}

\def\half{\nicefrac{1}{2}}

\def\is{\equiv}

\def\Re{{\rm Re}}
\def\Im{{\rm Im}}

\definecolor{lightgrey}{rgb}{0.8, 0.84, 0.8}

\newtheorem{theorem}[subsection]{Theorem}

\newtheorem{lemma}[subsection]{Lemma}

\newtheorem{corollary}[subsection]{Corollary}

\newtheorem{proposition}[subsection]{Proposition}
\newtheorem{remark}[subsection]{Remark}
\newtheorem{example}[subsection]{Example}

\title{$\Gamma$-evaluations of hypergeometric series}
\author{Frits Beukers, Jens Forsg\r{a}rd} 

\address{Utrecht University }
\email{f.beukers@uu.nl, jensforsgard@gmail.com}

\thanks{
This work was supported by the Netherlands Organisation for Scientific Research (NWO),
grant TOP1EW.15.313.}

\begin{document}
\maketitle

\section{Introduction}
Let $a,b,c\in\bbbc$ such that $c\not\in\bbbz_{\le0}$. The Gauss hypergeometric function
$F(a,b,c\,|\,z)$ is defined by the power series expansion
$$\sum_{n=0}^\infty\frac{(a)_n(b)_n}{(c)_n n!}\ z^n.$$
This power series converges in the complex disc $\,|\,z\,|\,<1$. When $\Re(c-a-b)>0$ the series
also converges on $\,|\,z\,|\,=1$. Note that if the $a$ or $b$ parameter is a negative integer
then $F(a,b,c\,|\,z)$ is a polynomial. There are no convergence issues in that case.

In the classical literature on hypergeometric functions we find many instances of special
evaluation of a hypergeometric function at specific arguments. The best known evaluation
is due to Gauss,
$$F(a,b,c\,|\,1)=\frac{\Gamma(c)\Gamma(c-a-b)}{\Gamma(c-a)\Gamma(c-b)}.$$
The left hand side converges only if $\Re(c-a-b)>0$. Another example is Kummer's evaluation
$$F(a,b,a-b+1\,|\,-1)=\frac{1}{2}\frac{\Gamma(\nicefrac{a}{2})\Gamma(a-b+1)}{\Gamma(a)\Gamma(\nicefrac a2-b+1)}.$$
From this one can deduce two others, as shown by Bailey in \cite[p.\ 11]{Ba}. The first is
$$F\big(2a,2b,a+b+\nicefrac12\,|\,\nicefrac12\big)=\frac{\Gamma(\nicefrac12)\Gamma(a+b+\nicefrac12)}{\Gamma(a+\nicefrac12)\Gamma(b+\nicefrac12)},$$
attributed to Gauss, and the second is
$$F\big(a,1-a,c\,|\,\nicefrac12\big)=\frac{\Gamma(\nicefrac c2)\Gamma(\nicefrac {(c+1)}2)}
{\Gamma(\nicefrac{(c+a)}2)\Gamma(\nicefrac{(1+c-a)}2)}.$$
There is a related evaluation
$$F(2a+1,b,2b\,|\,2)=\frac{\Gamma(-a)\Gamma(\nicefrac12+b)}{\Gamma(\nicefrac12)\Gamma(-a+b)}
\times\frac{1-e^{2\pi ia}}{2}.$$
However, for the moment this is only well-defined when $2a+1\in\bbbz_{\le0}$
and $2b\not\in\bbbz_{\le0}$, since in that case the left hand side is a finite sum.

The above examples contain 3 or 2 degrees of freedom in their parameters.
It turns out that there exists a very extensive list of one parameter evaluations.
As an example we quote from Bateman's \cite[2.8(53)]{erdelyi},
$$F\big(-a,-a+\nicefrac12,2a+\nicefrac32\,|\,-\nicefrac13\big)=\left(\frac{8}{9}\right)^{2a}\frac{\Gamma(2a+\nicefrac32)\Gamma(\nicefrac43)}
{\Gamma(2a+\nicefrac43)\Gamma(\nicefrac32)}.$$
These evaluations take place at fixed arguments and the values are a product of values of
$\Gamma$-functions times an exponential function times, possibly, a periodic function
in the hypergeometric parameters. In the literature they are sometimes
called "strange evaluation", we prefer the more descriptive name
$\Gamma${\it -evaluations}. 

The first systematic study that we are aware of is from W.Heyman
in 1899, \cite{Hey}. There we find a collection of $\Gamma$-evaluations
obtained by using the contiguity property for hypergeometric functions.
We also cite \cite{GS} from 1982 and \cite{S98} from 1998, which includes 
special evaluations for higher order hypergeometric functions as well.
The evaluations are often in polynomial form, by which we 
mean that one of the first two hypergeometric parameters is a negative integer.
The development of computer algebra methods made it possible
to automatize the search for $\Gamma$-evaluations.
See for example \cite{Ge} and the remarkable manuscript
\cite{ekhad} containing 40 $\Gamma$-evaluations discovered around 2004 by Shalosh Ekhad,
Doron Zeilberger's tireless computer. One more or less random example of such a $\Gamma$-evaluation is 
$$F\big(2t,t+\nicefrac13,\nicefrac43\,|\,-8\big)=\frac{2\cos\big(\pi(t+\nicefrac13)\big)}{27^t}\frac{\Gamma(t-\nicefrac16)\Gamma(\nicefrac12)}{\Gamma(t+\nicefrac12)\Gamma(-\nicefrac16)}$$
which can be found in \cite[(3.7)]{GS} when $t\in-\nicefrac13+\bbbz_{\le0}$ and additionally in 
\cite[4.3.2(xxi)]{ebisu} when $2t\in\bbbz_{\le0}$.
In this paper we show that it holds for arbitrary $t$.

The inspiration for the present paper comes from Akihito Ebisu's remarkable AMS Memoir \cite{ebisu}, in which the author develops a systematic method to find $\Gamma$-evaluations of Gaussian hypergeometric functions. As in Heymann's work,
the main tool in this study is the contiguity property of hypergeometric
functions. After explanation of this
idea, Ebisu produces a long list of sample $\Gamma$-evaluations, either in the finite form,
with one $a$ or $b$ parameter in $\bbbz_{\le0}$, or an interpolated version which holds fo
all parameter values $t$. We have adapted Ebisu's approach, which very briefly comes down to the following.

Consider a triple of hypergeometric parameters $a,b,c$ and abbreviate it by $\beta:=(a,b,c)$. We denote
$F(\beta\,|\,z):=F(a,b,c\,|\,z)$. Let $k,l,m$ be a triple of integers, which we denote as $\gamma:=(k,l,m)$,
the {\it shift vector}. Using contiguity relations we can find
rational functions $R_\gamma(\beta,z)$ and $Q_\gamma(\beta,z)$ in $\bbbq(a,b,c,z)$
such that
$$F(\beta+\gamma\,|\,z)=R_\gamma(\beta,z)F(\beta\,|\,z)+Q_\gamma(\beta,z)F'(\beta\,|\,z).$$
A quadruple $(\beta,z_0):=(a,b,c,z_0)$ is called {\it admissible} with respect to $\gamma$
if $Q_\gamma(\beta+t\gamma,z_0)=0$
for all $t\in\bbbc$. 
Choose an admissible quadruple $(\beta,z_0)$. We then obtain the functional equation
\begin{equation}\label{functional}
F(\beta+(t+1)\gamma\,|\,z_0)=R_\gamma(\beta+t\gamma,z_0) F(\beta+t\gamma\,|\,z_0)
\end{equation}
for $F(\beta+t\gamma\,|\,z_0)$ as function of $t$. Suppose that 
$$R_\gamma(\beta+t\gamma,z_0)=R_0\prod_{i=1}^r\frac{t+\alpha_i}{t+\delta_i},\quad R_0\in\bbbc^\times.$$
Then observe that $R_0^t\prod_{i=1}^r\frac{\Gamma(t+\alpha_i)}{\Gamma(t+\delta_i)}$
satisfies the same functional equation as $F(\beta+t\gamma\,|\,z_0)$. All we need to
do is identify these two functions of $t$.
This is done in Theorem \ref{main}, which is our main result.
From Theorem \ref{main} we can deduce
interpolated versions of $\Gamma$-evaluations which occured only in finite
form in earlier publications.

Since in many of the latter cases the argument is outside the
disc of convergence we need to extend the evaluations of $F(a,b,c\,|\,z)$ to $z$
outside the unit disc. 

The sum $F(a,b,c\,|\,z)$ can be continued analytically to $\bbbc\setminus[1,\infty)$
using Euler's integral
$$
F(a,b,c\,|\,z)=\frac{\Gamma(c)}{\Gamma(b)\Gamma(c-b)}\int_0^1\frac{x^{b-1}(1-x)^{c-b-1}}{(1-zx)^{a}}\ dx.
$$
In the integrand we choose $x^b=\exp(b\log|x|)$ and $(1-x)^{c-b}=\exp((c-b)\log|1-x|)$,
and we define $(1-zx)^a$ using the choice $|\arg(1-zx)|<\pi$.
Note that this integral only converges at the points $0$ and $1$ if $\Re(b)$ and $\Re(c-b)$
are positive. To get an integral without these restrictions
one can replace the path of integration $[0,1]$ by the so-called Pochhammer
contour $C$:

\centerline{\includegraphics{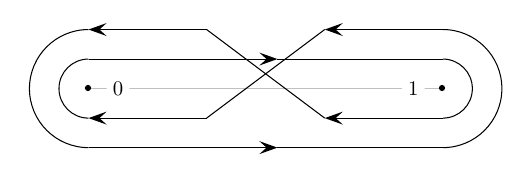}}

and division of the integral by $(e^{-2\pi ib}-1)(e^{2\pi i(c-b)}-1)$.
The four
horizontal piecewise linear paths should be thought of as four copies of the real segment 
$[\delta,1-\delta]$ and the rounded parts as the circles $|z|=\delta$ and $|z-1|=\delta$
for some small $\delta>0$. 
We have taken the argument of the integrand on the bottom line segment to
be given as above. 
For the evaluation of $F(a,b,c\,|\,z)$ at $z\in(1,\infty)$ we make the choice 
$\lim_{\epsilon\downarrow0}F(a,b,c\,|\,z+\epsilon i)$.
Its value is now given by the Euler integral over the arc

\centerline{\includegraphics{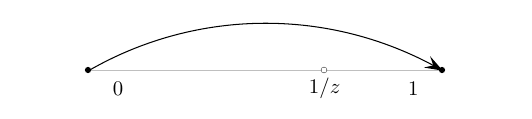}}

or its Pochhammer version. When $a,b,c$ are real,
the value $\lim_{\epsilon\downarrow0}F(a,b,c\,|\,z-\epsilon i)$ is its complex conjugate. When $a,b,c$ are not all real, the difference between these limits
can be quite drastic.
The reader should be aware of this when checking the results numerically.
For example, the computer package {\it Mathematica}
seems to use the second limit (with $z-i\epsilon$). 

In the above description we have suggested that the degrees in $t$
of numerator and denominator
of $R_\gamma(\beta+t\gamma,z_0)$ are the same.
In Theorem \ref{belyi} we prove that this is indeed the case
when the vector $\beta+t\gamma$ is non-resonant. This means that none of the four
linear functions 
\[
a+kt,\quad b+lt,\quad c-a+t(m-k),\quad \text{and} \quad c-b+t(m-l)
\]
is an integer valued constant.
It turns out that the non-resonant case is the interesting case;
in Section \ref{resonant} we
give a description of the resonant cases only for completeness.
In the non-resonant case, Theorem \ref{belyi} also gives the values of $z_0$ and $R_0$.
This is a result found previously by Iwasaki in
\cite[Thm 2.3]{Iw}, although not in this wording and with a different proof
using asymptotic analysis of the Euler integral.

Although we believe that for a given admissable quadruple there should exist a simple
procedure to determine $R_\gamma(\beta+\gamma t)$, we have not been able to discover it.
Another issue we should mention is a difference between the result of Theorem \ref{main}
and some finite evaluations in \cite{ekhad} and \cite{ebisu}. As an example consider
the identity 
$$F(t,3t-1,2t|e^{\pi i/3})=-\frac{\sqrt{3}}{2}e^{\pi i(\nicefrac{t}2+\nicefrac56)}
\left(\frac{4}{\sqrt{27}}
\right)^t\frac{\Gamma(t+\nicefrac12)\Gamma(\nicefrac13)}
{\Gamma(t+\nicefrac13)\Gamma(\nicefrac12)},$$\label{anexample}
which can be deduced from Theorem \ref{main}. It holds for all $t\in\bbbc$.
When $t=-n$ for any $n\in\bbbz_{>0}$ the left hand side is not well-defined
as hypergeometric series, but the equality should be read as the limit when
$t\to-n$. We get, after some simplication, 
$$F(-n,-3n-1,-2n|e^{\pi i/3})=-\frac{\sqrt{3}}{2}e^{5\pi i/6}
\left(\frac{-\sqrt{-27}}{4}\right)^{n}\frac{(\nicefrac23)_n}{(\nicefrac12)_n}.$$
In \cite[Thm 11]{ekhad} and \cite[4.2.4]{ebisu} we find the same evaluation,
but with the factor $-\sqrt{3}e^{5\pi i/6}/2$ missing. The reason is that in
the latter evaluations the function $F(-n,-3n-1,-2n|z)$ is interpreted as the
polynomial $F(-n,-3n-1,c|z)|_{c=-2n}$. It is remarkable that the limit and the
polynomial evaluation differ by a constant factor. In many cases when both the
$a$-parameter and $c$-parameter have limits that 
are non-positive integers this phenomenon seems to occur. 
As suggested by the referee, an explanation might be that both sequences
satisfy the same first order recurrence relation in $n$. We have not tried to elaborate this.

We have not made an exhaustive search for all admissable quadruples. This is more
or less done in \cite{ebisu}. There it is also remarked that through
the use of Kummer's solution to a hypergeometric equation
any admissable quadruple is associated to 24 others. This may explain the
abundance of these $\Gamma$-evaluations. In Section \ref{Euler}
we give a description and a proof of the existence
of these associated quadruples through the properties of the Euler kernel, which
is the integrand of the Euler integral. 

In the final section we present a more or less random list of examples of 
$\Gamma$-evaluations.

{\bf Acknowledgement} We like to thank the referee extensively for his (or her)
careful reading of the manuscript, for the correction of several 
large and small errors and for the suggestions for improvement. 

\section{Interpolation}\label{interpolation}

Let us begin with an example. We consider the case $(k,l,m)=(2,2,1)$ and carry out
the program we sketched in the introduction. We get
\begin{eqnarray*}
R_\gamma(a,b,c,z)&=&\frac{c(2+a+b-c)}{(a+1)(b+1)(z-1)^2} \quad\text{and}\\
Q_\gamma(a,b,c,z)&=&\frac{c((1 + 2 a + a^2 + 2 b + a b + b^2 - c - a c - b c)z+(1+a-c)(1+b-c))}
{a (1 + a) b (1 + b) (-1 + z)^2}.
\end{eqnarray*}
The numerator of $Q_\gamma(a+2t,b+2t,c+t,z)$ reads
\begin{eqnarray*}
&&(c+t)\left(1 + a + b + a b - 2 c - a c - b c + c^2 + z + 2 a z + a^2 z + 2 b z + 
 a b z + b^2 z\right.\\
&&\left.- c z - a c z - b c z + (2 + a + b - 2 c + 7 z + 5 a z + 5 b z - 4 c z)t + (8z+1)t^2\right).
\end{eqnarray*}
The equations for the admissible quadruple are obtained by setting this polynomial in
$t$ identically zero. We get
\[
\left\{
\begin{array}{lll}
0&=&8z+1\\
0&=&2 + a + b - 2 c + 7 z + 5 a z + 5 b z - 4 c z\\
0&=&1 + a + b + a b - 2 c - a c - b c + c^2 + z + 2 a z + a^2 z \\
&& + 2 b z + a b z + b^2 z - c z - a c z - b c z.
\end{array}\right.
\]
Solution of this system yields
\begin{equation}
\label{eq:Section2FirstExample}
z_0=-\nicefrac18,\quad a=2t,\quad b=2t+\nicefrac13,\quad c=t+\nicefrac56
\end{equation}
or 
$$z_0=-\nicefrac18,\quad a=2t,\quad b=2t-\nicefrac13,\quad c=t+\nicefrac23.$$
Taking the first possibility we get
$$R_\gamma\big(2t,2t+\nicefrac13,t+\nicefrac56,-\nicefrac18\big)=\frac{16}{27}\times\frac{t+\nicefrac56}{t+\nicefrac23}.$$
So we find from (\ref{functional}) that
$$F\big(2(t+1),2(t+1)+\nicefrac13,t+1+\nicefrac56\,|\,-\nicefrac18\big)=\frac{16}{27}\times\frac{t+\nicefrac56}{t+\nicefrac23}
\times F(2t,2t+\nicefrac13,t+\nicefrac56\,|\,-\nicefrac18)$$
for all $t$. Notice that $\left(\frac{16}{27}\right)^t\frac{\Gamma(t+\nicefrac56)}{\Gamma(t+\nicefrac23)}$
satisfies the same functional equation. The corresponding functions
turn out to differ by a constant factor, as shown in Corollary 
\ref{example1}.

In this section we prove Theorem \ref{main} which states that for
any admissible quadruple $(\beta,z_0)$ there
exists a complex interpolation of the $\Gamma$-evaluations. 
We find from \cite[Corollary 1.4.4]{AAR} the following estimate. 
\begin{lemma}\label{gamma-estimate}
Suppose $s=a+bi$ with $a_1<a<a_2$ and $\,|\,b\,|\,\to\infty$.
Then
$$ | \Gamma(a+bi) | = \sqrt{2\pi} \,|b|^{a-\frac12}\,e^{-\frac{\pi | b |}{2}}\, \big[1+O(1/ | b | )\big].$$
\end{lemma}

\begin{proposition}\label{hypergeo-estimate}
Let $\beta=(a,b,c)\in\bbbr^3$ and $\gamma=(k,l,m)\in\bbbz^3$. Let $z_0\in\bbbc$ and
$z_0\ne1$. Then, 
$F(\beta+\gamma t \,| z_0)$ is a meromorphic function in $t\in\bbbc$ having at most finitely
many poles with $ | \Re(t) | \le\frac12$. Let 
$$C_1= | k\arg(1-z_0) | + \frac{| l | \pi}2+ \frac{|m-l | \pi}2- \frac{| m | \pi}2,\quad  |\arg(1-z_0) | \le\pi.$$
Then there exist $C_2,C_3\ge0$ such that
$$ \big| F(\beta+\gamma t\, | z_0) \big| 
\le C_2\, | \Im(t) | ^{C_3} \,e^{C_1 | \Im(t) | }$$
for all $t\in\bbbc$ with $  |\Re(t)|  \le \frac12$ and $ | \Im(t) | $ sufficiently large.
\end{proposition}

\begin{proof}
In order to prove our estimate we use the ordinary Euler integral.
We first prove the proposition under the assumption that $-a> \frac{| k |}2,$ that $ b> \frac{| l |}2$, and 
that $c-b> \frac{| m-l |} 2$.
Let us write
$$G(\beta+\gamma t \,| z_0):=
\int_0^1\frac{x^{b-1+lt}(1-x)^{c-b-1+(m-l)t}}{(1-z_0x)^{a+kt}} dx.$$
This integral converges for all $t$ with $|\Re(t)|\le\frac12$ because of our assumptions on
$a,b,$ and $c$. Since $|\Re(t)|\le\frac12$, we get
$$ | x^{b-1+lt} | \le x^{b-1- \frac{| l |}2}\ \quad \mbox{and}\quad 
\  | (1-x)^{c-b-1+(m-l)t} | \le (1-x)^{c-b-1- \frac{| m-l |}2}.$$
Let $c_1=\int_0^1x^{b-1-\frac{ | l |}2}(1-x)^{c-b-1- \frac{| m-l |}2}dx$. Recall that
$$
 \big| (1-z_0x)^{-a-kt}\big | = | 1-z_0x | ^{-a-k\Re(t)}\exp\big(k\arg(1-z_0x)\Im(t)\big).
$$
Let $c_2=\max_{x\in[0,1], | y | \le\frac12} | 1-z_0x | ^{-a-ky}$, which
is finite because $-a>|k|/2$. Notice also that
$$
\max_{x\in[0,1]}\exp\big(k\arg(1-z_0x)\Im(t)\big)
\le \exp\big( |k \arg(1-z_0)\,\Im(t) | \big).
$$
We conclude that $ | G(\beta+t\gamma | z_0) | $ has the upper bound
$c_1c_2\,e^{ | k\arg(1-z_0)\,\Im(t) | }$. Using Lemma~\ref{gamma-estimate} we
find the desired estimate for 
$$F(\beta+t\gamma\, | z_0)=\frac{\Gamma(c+mt)}
{\Gamma(b+lt)\Gamma(c-b+(m-l)t)}G(\beta+t\gamma | z_0)$$
when $-a> \frac{|k|}2,b> \frac{|l|}2$ and $c-b> \frac{|m-l|}2$. 

In the general situation we first choose integers $\Delta a,\Delta b,\Delta c$ such that
$$
-a-\Delta a-1> \frac{| k |}2,\quad b+\Delta b> \frac{|l|}2, \quad c-b+\Delta c-\Delta b> \frac{|m-l|}2.
$$ 
Denote $\Delta\beta=(\Delta a,\Delta b,\Delta c)$. Then, there exists a contiguity relation
$$
F(\beta+t\gamma | z)=r(t,z)F\left(\beta+\Delta\beta+t\gamma | z\right)
+s(t,z)F\left(\beta+\Delta\beta+(1,1,1)+t\gamma | z\right),
$$
where $r(t,z)$ and $s(t,z)$ are rational functions in $z,t$.
In Lemma \ref{contiguity-denominator} we show that, as rational function of $z$,
their only poles are in $z=0,1$.Hence we can specialize to $z=z_0$ and get
$$
F(\beta+t\gamma | z_0)=r(t,z_0)F\left(\beta+\Delta\beta+t\gamma | z_0\right)
+s(t,z_0)F\left(\beta+\Delta\beta+(1,1,1)+t\gamma | z_0\right),
$$
We then apply the above estimate to
the terms on the right hand side. 
\end{proof}

\begin{lemma}\label{contiguity-denominator}
Let $a,b,c$ be hypergeometric parameters such that $a,b,c-a,c-b\not\in\bbbz$.
Let $a',b',c'$ be contiguous parameters, that is $a'-a,b'-b,c'-c\in\bbbz$.
Consider the contiguity relation
$$
F(a',b',c'|z)=r(z)F(a,b,c|z)+s(z)F(a+1,b+1,c+1|z),
$$
where $r(z),s(z)$ are rational functions in $a,b,c,z$. Then as rational
functions in $z$, the functions $r(z),s(z)$ have only poles in $z=0,1$.
\end{lemma}

\begin{proof}
We know that $F(a+1,b+1,c+1|z)=\frac{c}{ab}F'(a,b,c|z)$. The contiguity
relation is stable under analytic continuation in $z$.
Therefore we have a similar relation for the solution of the
hypergeometric equation corresponding to the local exponent $1-c$. Thus there
exists a rational function $\lambda$ in $a,b,c$ such that
$\lambda z^{1-c'}F(a'+1-c',b'+1-c',2-c'|z)$ equals
$$r(z)z^{1-c}
F(a+1-c,b+1-c,2-c|z)+s(z)\frac{c}{ab}(z^{1-c}F(a+1-c,b+1-c,2-c|z))'.
$$
We can now solve for $r(z),s(z)$ and find that
\begin{align*}
\begin{pmatrix}r(z)\\ \frac{c}{ab}s(z)\end{pmatrix}=&
\frac{1}{W(z)}\begin{pmatrix} 
(z^{1-c}F(a+1-c,b+1-c,2-c|z)' & - F'(a,b,c|z)\\
-z^{1-c}F(a+1-c,b+1-c,2-c|z) & F(a,b,c|z)
\end{pmatrix}\\
&\times\begin{pmatrix}F(a',b',c') \\ 
\lambda z^{1-c'}F(a'+1-c',b'+1-c',2-c'|z)\end{pmatrix},
\end{align*}
where $W(z)$ is the Wronskian determinant of the hypergeometric equation,
which equals $1-c$ times $z^{-c}(1-z)^{c-a-b-1}$. The matrices on the
right hand side have entries which are locally holomorphic outside $0,1,\infty$
and therefore we conclude that the same holds for $r(z),s(z)$. 

Strictly speaking we have proved the lemma when $c\not\in\bbbz$. The case
of integral $c$ runs similarly.
\end{proof}

\begin{proposition}\label{minicarlson}
Let $f(t)$ be a periodic entire function with unit period one. Suppose that
there are constants $C^+,C^-\ge0$ such that 
\begin{enumerate}
\item $ | f(t) | =O\big(e^{C^+\Im(t)}\big)$ when $\Im(t)\to\infty$, and 
\item $ | f(t) | =O\big(e^{-C^-\Im(t)}\big)$ when $\Im(t)\to-\infty$.
\end{enumerate}
Then, $f(t)=g(e^{2\pi it})$ where $g(z)\in\bbbc[z,1/z]$.
Moreover, $g$ has a pole of order at most $C^+/2\pi$ at $z=0$, and
a pole of order at most $C^-/2\pi$ at $z=\infty$.
\end{proposition}

\begin{proof}
Consider the composite function $g(z)=f\big(\frac{\log z}{2\pi i}\big)$. This is an
entire function in $z$, except possibly
at $z=0$, which is an isolated singularity. Notice that $\Im(t)=-\frac{\log | z |}{ 2\pi}$.
So when $z\to0$ we get $\Im(t)\to\infty$ and we can use the estimate
$$| f(t) | =O\big(e^{-\frac{\log | z |}{ 2\pi}\,C^+}\big)=O\big( | z | ^{-\frac{C^+}{2\pi}}\big).$$
When $z\to\infty$ we get $\Im(t)\to-\infty$ and we can use the estimate
\[
| f(t) | =O\big(e^{\frac{\log | z |}{2\pi}\,C^-}\big)=O\big(|z|^{\frac{C^-}{2\pi}}\big).\qedhere
\]
\end{proof}

We can now show our main theorem.

\begin{theorem}\label{main}
We use the notations from the introduction. Let $(\beta,z_0)$ be an admissible
quadruple with respect to $\gamma=(k,l,m)\in\bbbz^3$. We assume that $m\ge0$ and
$c\not\in\bbbz_{\le0}$ when $m=0$. Write
$$R_\gamma(\beta+t\gamma,z_0)=R_0\times\prod_{j=1}^r\frac{(t+\alpha_j)}
{(t+\delta_j)}.$$
Then, there exists $g(z)\in\bbbc[z,1/z]$ such that
$$
F(\beta+t\gamma \,|\,z_0)=g\left(e^{2\pi it}\right)\,R_0^t\prod_{j=1}^r\frac{\Gamma(t+\alpha_j)}
{\Gamma(t+\delta_j)}$$
for all $t\in\bbbc$. Moreover, $g$ has a pole order at most
$$\frac{\arg(R_0)}{2\pi}+ \frac{| k\arg(1-z_0) |}{ 2\pi}+ \frac{| l |}4+ \frac{| m-l |} 4- \frac{| m |}4$$
at $z=0$ and order at most
$$-\frac{\arg(R_0)}{2\pi}+ \frac{| k\arg(1-z_0) |}{2\pi}+ \frac{| l |} 4+ \frac{| m-l |}4- \frac{| m |}4$$
at $z=\infty$.
\end{theorem}

\begin{remark}
We have used that the numerator and denominator of $R$ have the same
degree. This is a consequence of Lemma~\ref{lemma:DegreesInt}. 
\end{remark}

\begin{remark}
The assumption $m\ge0$ is not a restriction. If $m<0$, then we apply Theorem~\ref{main}
with $-\gamma$ and simply replace $t$ by $-t$.
\end{remark}

\begin{proof}[Proof of Theorem~\ref{main}.]
We find that
$$G(t):=F(\beta+t\gamma\,|\,z_0)\,R_0^{-t}\,\prod_{j=1}^r\frac{\Gamma(t+\delta_j)}{\Gamma(t+\alpha_j)}$$
is a meromorphic periodic function with period $1$. Poles can only arise from 
the factor
$F(\beta+t\gamma\,|\,z_0)$ when $c+mt\in\bbbz_{\le0}$, or from the product 
$\prod_j\Gamma(t+\delta_j)$ when
$t+\delta_j\in\bbbz_{\le0}$ for some $j$. It follows, since $m \ge 0$, that 
there are no poles when $\Re(t)$ is sufficiently large. Hence, $G(t)$ is holomorphic in $t$. 
We now use the estimates from Lemma~\ref{gamma-estimate} and 
Proposition~\ref{hypergeo-estimate} to get $ |R_0^t\,G(t) | =O(e^{(C_1+\epsilon) | \Im(t) | })$ for any $\epsilon>0$, where 
$$C_1= | k\arg(1-z_0) | + \frac{| l |\pi}2+ \frac{ | m-l |\pi} 2- \frac{| m |\pi}2,$$
as in Proposition \ref{hypergeo-estimate}. This yields $ | G(t) | =O(e^{(\arg(R_0)+C_1+\epsilon) | \Im(t) | })$ when $\Im(t)\to
\infty$ and $ | G(t) | =O(e^{(-\arg(R_0)+C_1+\epsilon) | \Im(t) | })$ when $\Im(t)\to-\infty$. 
The result now follows from Proposition~\ref{minicarlson}.
\end{proof}

We give three example applications.

\begin{corollary}\label{example1}
For all $t\in\bbbc$ we have 
$$
F\big(2t,2t+\nicefrac13,t+\nicefrac56 \,|\,-\nicefrac18\big)=\left(\frac{16}{27}\right)^t\frac{\Gamma(t+\nicefrac56)\Gamma(\nicefrac23)}
{\Gamma(t+\nicefrac23)\Gamma(\nicefrac56)}.
$$
\end{corollary}

\begin{proof}
In the beginning of this section we considered the example $\gamma=(2,2,1)$ and the admissible
quadruple \eqref{eq:Section2FirstExample}. From Theorem \ref{main},
applied to this example, we find that
$$\left(\frac{27}{16}\right)^t\frac{\Gamma(t+\nicefrac23)}{\Gamma(t+\nicefrac56)}F(2t,2t+\nicefrac13,t+\nicefrac56\,|\,-\nicefrac18)$$
is a Laurent polynomial in $e^{2\pi it}$. Since $\arg(16/27)=\arg(1-z_0)=0$ the estimates
for the pole order of $g$ at $0$ and $\infty$ are $\frac12$. Hence, $g$ is constant. The value of the constant can be found by setting $t=0$.
\end{proof}

\begin{corollary}\label{example2}
For all $t\in\bbbc$ we have
$$
F\big(3t,t+\nicefrac16,\nicefrac12\,|\,-3\big)=\frac{\cos(\pi t)}{16^t}\frac{\Gamma(t+\nicefrac12)\Gamma(\nicefrac13)}{\Gamma(t+\nicefrac13)\Gamma(\nicefrac12)}.$$
\end{corollary}

\begin{proof}
Consider the admissible quadruple $a=3t,b=t+\nicefrac16,c=\nicefrac12,$ and $z_0=-3$. We get
$$R_\gamma(\beta+t\gamma,z_0)=-\frac{1}{16}\times\frac{t+\nicefrac12}{t+\nicefrac13}.$$
Application of Theorem \ref{main} yields
$$
F(3t,t+\nicefrac16,\nicefrac12\,|\,-3)=\frac{e^{\pi it}}{16^t}\frac{\Gamma(t+\nicefrac12)}{\Gamma(t+\nicefrac13)}g\left(e^{2\pi it}\right).$$
Here, $g(z)$ is a Laurent polynomial, bounded at $z = \infty$, and with a pole at $z=0$ of order at most $1$.
Hence, $g\big(e^{2\pi it}\big)=u+ve^{-2\pi it}$ for some $u,v\in\bbbc$. 
Setting $t=0$ and $t=-\frac12$ yields
\[
\left\{\begin{array}{lll}
1&=&(u+v)\Gamma(\nicefrac12)/\Gamma(\nicefrac13)\\
0 & = &  u - v.
\end{array}\right.
\]
Hence, $u=v=\Gamma(\nicefrac13)/2\Gamma(\nicefrac12)$ and our corollary follows.
\end{proof}

\begin{corollary}\label{example3}
For all $t\in\bbbc$ we have
$$F\big(3t,t+\nicefrac16,\nicefrac12\,|\,9\big)=\frac{1}{2\cdot 64^t}\left(1+e^{2\pi i\left(t+\frac16\right)}-e^{4\pi i\left(t+\frac16\right)}\right).$$
\end{corollary}

\begin{proof}
Consider the admissible quadruple $a=3t,b=t+\nicefrac16,c=\nicefrac12$, and $z_0=9$.
We find that $R_\gamma=\frac1{64}$.
So Theorem \ref{main} gives $F(3t,t+\nicefrac16,\nicefrac12\,|\,9)=64^{-t}g\left(e^{2\pi it}\right)$.
Since $|\arg(1-9)|=\pi$, we get the estimate $2$ for the polar order of $g(z)$ at $z=0$ and at $z=\infty$.
Hence, $g\big(e^{2\pi it}\big)=\sum_{k=-2}^2a_ke^{2\pi ikt}$. To determine the values of the $a_k$ we
use five special evaluations of $64^tF(3t,t+\nicefrac16,\nicefrac12\,|\,9)$
for $t = 0, -\frac13, -\frac23, -\frac16,$ and $\frac16$.
We obtain the system
\[
\left[\begin{array}{c}
1 \\ 1 \\ -\frac12 \\ \frac12 \\ \zeta
\end{array}\right]
=
\left[\begin{array}{ccccc}
 1 & 1 & 1 & 1& 1\\
  \zeta^{-2} &\zeta^2&1&\zeta^{-2}&\zeta^2  \\
  \zeta^2 &\zeta^{-2}&1 &\zeta^2&\zeta^{-2} \\
 \zeta^2&\zeta&1&\zeta^{-1}&\zeta^{-2} \\
 \zeta^{-2}&\zeta^{-1}&1&\zeta&\zeta^2
\end{array}\right]
\left[\begin{array}{l}
a_{-2} \\ a_{-1} \\ a_0 \\ a_1 \\ a_2
\end{array}\right]
\]
where $\zeta = e^{\frac{\pi i}3}$.
The evaluation at $t=1/6$ requires some explanation. We need
to determine $2F(\nicefrac12,\nicefrac13,\nicefrac12|9)=
2\cdot(1-9)^{-\nicefrac13}$.
The absolute value is of course $1$. It remains to determine the argument.
By our branch choice we take a path in the upper half plane from $z=1/2$
to $z=9$. The argument of $(1-z)^{-\nicefrac13}$ then changes from $0$ to
$\pi/3$. Hence the function value becomes $\zeta$.
Solution of the system gives
our corollary.
\end{proof}

\section{Resonant quadruples}\label{resonant}
Our next goal is to explain the values of $z_0$ and $R_\gamma(\beta+t\gamma)$
that occur in the above considerations. For that purpose it turns out to be
convenient to restrict
to admissible quadruples such that $\beta+t\gamma$ is {\it non-resonant},
that is, none of 
\[
a+kt,\quad b+lt,\quad c-b+(m-l)t,\quad \text{ and } \quad c-a+(m-k)t
\] is an element of $\bbbz$. If
at least one of these linear polynomials is an integer constant, then we say that the quadruple
is {\it resonant}. In this section we make some comments on the resonant case
and proceed with the non-resonant cases in the next sections.
We will use the identity
\begin{equation}\label{secondfunction}
F(c-a,c-b,c\,|\,z)=(1-z)^{a+b-c}F(a,b,c\,|\,z).
\end{equation}

Suppose that $\beta+t\gamma$ is resonant and that $(\beta+t\gamma,z_0)$
is a resonant quadruple. Then, we distinguish the following
cases.

\begin{enumerate}
\item Exactly one of $a+kt,b+lt,c-b+(m-l)t,$ and $c-a+(m-k)t$ is an integer.
\begin{enumerate}
\item \label{item:casea} If $a+kt\in\bbbz$, then we conjecture that the admissible quadruples
are given by $a=2,b=1+lt,c=2+mt,$ and $z_0=\nicefrac ml$,  with 
$\Gamma$-evaluation
$$F(2,1+lt,2+mt\,|\,m/l)=\frac{l(1+mt)}{l-m},$$
or $a=-1,b=lt,c=mt,$ and $z_0=\nicefrac ml$ with $\Gamma$ evaluation
$$F(-1,lt,mt\,|\,m/l)=0.$$
It should be remarked that these evaluations are a direct consequence of the
general identities
$$F\Big(2,r,s\,\Big|\,\frac{s-2}{r-1}\Big)=\frac{(r-1)(s-1)}{r-s+1}
\quad \text{ and } \quad 
F\Big(-1,r,s\,\Big|\,\frac sr\Big)=0,$$
which are easy to prove. A similar remark applies to the next cases.
\item The case $b+lt\in\bbbz$ is similar to case \eqref{item:casea}.
\item  \label{item:casec} If $c-a+(m-k)t\in\bbbz$, then we use the identity
(\ref{secondfunction}) to get
$$F\big(mt,1+(m-l)t,2+mt\,|\,\nicefrac ml\big)=\left(1-\frac{m}{l}\right)^{(l-m)t}(1+mt)$$
and
$$F\big(mt+1,(m-l)t,mt\,|\,\nicefrac ml\big)=0.$$
\item The case $c-b+(m-l)t\in\bbbz$ is similar to case \eqref{item:casec}.
\end{enumerate}

\item Exactly two of $a+kt,b+lt,c-b+(m-l)t,c-a+(m-k)t$ are in $\bbbz$,
\begin{enumerate}
\item If $a+kt, b+lt \in \bbbz$, then admissability implies that either
$ab=0$ or 
$a\in\bbbz,b=1-a,c=mt,$ and $z_0=\nicefrac12$. In the latter
case we might as well replace $mt$ by $t$. Bailey's identity gives
$$F(a,1-a,t\,|\,\nicefrac12)=\frac{\Gamma(\nicefrac t2)\Gamma(\nicefrac{(t+1)}2)}{\Gamma(\nicefrac{(t+a)}2)\Gamma(\nicefrac{(1+t-a)}2)}.$$
\item If $a+kt,c-b+(m-l)t\in\bbbz$, then admissability implies that
either $a(b-c)=0$ or $a\in\bbbz,b=t,c=t-a+1,$ and $z_0=-1$.
In the latter case Kummer's identity gives
$$F(a,t,t+1-a\,|\,-1)=\frac{1}{2}\frac{\Gamma(\nicefrac t2)\Gamma(t-a+1)}{\Gamma(t)\Gamma(\nicefrac t2-a+1)}.$$
\item If $b+lt,c-b+(m-l)t\in\bbbz$, then admissability implies that
either $b=1$ and $c=2$, or $b\in\bbbz,c=2b,$ and $z_0=2$.
In the former case we get
$$F(1+t,1,2\,|\,z)=\frac{(1-z)^{-t} - 1}{tz},$$
in the latter case we get
$$F(1-2t,b,2b\,|\,2)=\frac{\Gamma(t)\Gamma(b+\nicefrac12)}{\Gamma(t+b)\Gamma(\nicefrac12)}
\times\frac{1 - e^{-2\pi it}}{2}.$$
\item The other three cases are related to the above three via the identity
(\ref{secondfunction}).
\end{enumerate}
\end{enumerate}

\section{Euler kernels}\label{Euler}
Let $\beta$ be the triple of hypergeometric parameters and $\gamma$ the shift vector as
in the previous section. Suppose also that $z\ne0,1$. We define
$$K(\beta,z,x)=\frac{x^{b-1}(1-x)^{c-b-1}}{(1-zx)^a}.$$
Application of the Pochhammer contour integral then gives us
$$\frac{\Gamma(b)\Gamma(c-b)}{\Gamma(c)}F(\beta\,|\,z).$$
In \cite{Be15} the author considered the $\bbbq(\beta,z)$-vector space of twisted differential forms
generated by the differential forms $K(\beta+\gamma,z,x)dx$ with $\gamma\in\bbbz^3$.
The $\bbbq(\beta,z)$-vector space of twisted exact forms is generated by 
$d(K(\beta+\gamma,z,x))$
with $\gamma\in\bbbz^3$. We denote the quotient space by $H^1_{\rm twist}(\beta\,|\,z)$. 
In \cite[Thm 6.1]{Be15} it is shown, under the assumption $\beta$ is non-resonant,
that this space is two dimensional with basis
$K(\beta,z,x)dx$ and $K(\beta+(1,1,1),z,x)dx$. Notice that 
$$aK(\beta+(1,1,1),z,x)=\frac{\partial}{\partial z}K(\beta,z,x).$$
Define the hypergeometric operator
$${\cal L}=z(z-1)\frac{\partial^2}{\partial z^2}+((a+b+1)z-c)\frac{\partial}
{\partial z}+ab.$$
We find that
${\cal L}(K(\beta,z,x))=0$ in $H^1_{\rm twist}(\beta\,|\,z)$. 
Since ${\cal L}$ commutes with the application of the Pochhammer contour,
and Pochhammer integration is zero on exact forms, we
recover the hypergeometric equation for $F(a,b,c\,|\,z)$.  

Let $(\beta,z_0)$ be an admissible quadruple with respect to $\gamma$ and suppose
it is non-resonant. Let $M$ be the field $M=\bbbq(\beta,z_0)$.
Define
$$\hat{R}(t)=\frac{(b+lt)_l(c-b+(m-l)t)_{m-l}}{(c+mt)_m}
R_\gamma(\beta+t\gamma,z_0).$$
Here $(x)_n=x(x+1)\cdots(x+n-1)$ 
if $n\ge0$ and $(x)_n=\frac{1}{(x-1)\cdots(x-\,|\,n\,|\,)}$ if $n<0$.
We can rewrite equation (\ref{functional}) in terms of Euler kernels as
\begin{equation}\label{functionalDR}
K(\beta+(t+1)\gamma,z_0,x)dx\is\hat{R}(t)K(\beta+t\gamma,z_0,x)dx
\end{equation}
in $H^1_{\rm twist}(\beta+t\gamma\,|\,z_0)$.
Define also
$$
g_\gamma(z,x)=\frac{x^l(1-x)^{m-l}}{(1-zx)^k}.
$$
Then, (\ref{functionalDR}) amounts to the statement that there
exists $W(t,x)\in M(t,x)$ such that
\begin{equation}\label{specialcontiguity}
g_\gamma(z_0,x)K(\beta+t\gamma,z_0,x)=\hat{R}(t)K(\beta+t\gamma,z_0,x)+\frac{\partial}
{\partial x}\left(W(t,x)K(\beta+t\gamma,z_0,x)\right).
\end{equation}

Define the denominator $d_\gamma(x)$ of $g_\gamma(z_0,x)$ by
$$
d_\gamma(x)=x^{(-l)^+}(1-x)^{(l-m)^+}(1-z_0x)^{k^+},$$
where $x^+=\max(0,x)$. The numerator $n_\gamma(x)$ is defined by $d_\gamma(x)
 g_\gamma(z_0,x)$.

Let us write
$$
W(t,x)=\frac{p(t,x)}{d_\gamma(x)} x(1-x)(1-z_0x),$$
where $p(t,x)$ is another rational function which will turn out to
be a polynomial in $x$. Then, after multiplication by $d_\gamma(x)$ and
division by $K(x)$, (\ref{specialcontiguity}) can be rewritten as
\begin{equation}\label{specialcontiguity2}
n_\gamma(x)-\hat{R}(t)d_\gamma(x)=\frac{\partial}{\partial x}\left(
p(t,x)x(1-x)(1-z_0x)\right)+q(t,x)p(t,x),
\end{equation}
where
$$
q(t,x):=x(1-x)(1-z_0x)\left(\frac{b-1+lt}{x}+\frac{c-b-1+(m-l)t}{x-1}
-\frac{a+kt}{x-1/z_0}-
\frac{1}{d_\gamma}\frac{\partial d_\gamma}{\partial x}\right).$$
This is the log-derivative of $K(\beta+t\gamma,z_0,x)/d_\gamma(x)$
times $x(1-x)(1-z_0x)$. Note that $q(t,x)$
is a polynomial in $x$ of degree at most $2$ and linear in $t$.
The coefficient of $x^2$ reads $z(c-a-2+t(m-k)-\deg_x(d_\gamma))$, which is non-zero as
a result of our non-resonance condition. Therefore $q(t,x)$ has degree $2$ in $x$.
The non-resonance condition also sees to it that $q(t,x)$ has no zeros in $\{0,1,1/z_0\}$.

We shall write $q(t,x)=q_1(x)+tq_0(x)$. Notice that 
\begin{equation}\label{def:q0}
q_0(x)=x(1-x)(1-z_0x)\frac{g_\gamma'(z_0,x)}{g_\gamma(z_0,x)}
=(m-k)z_0x^2+((k-l)z_0-m)x+l.
\end{equation}
In particular, $q_0(x)$ is non-trivial.

It follows from (\ref{specialcontiguity2}) that $p(t,x)$ has no poles
outside $x=0,1,1/z$. Suppose it has a pole of order $\delta>0$ at $x=0$. Then, looking
at the order $\delta$ poles on both sides of (\ref{specialcontiguity2}) we get
$0=(1-\delta)+(b+tl-1-n),$
where $n$ is the pole order at $x=0$ of $g_\gamma$. This implies $b+lt\in\bbbz$, contradicting
our non-resonance condition. Similarly we show that $p(t,x)$ has no poles in $x=1,1/z$.
Hence $p(t,x)$ is a polynomial in $x$. Its degree in $x$ turns out to be at most
$\max(\deg_x(d_\gamma),\deg_x(n_\gamma))-2$. For the latter fact we use the condition
$c-a+t(m-k)\not\in\bbbz$. 

We may interpret equation (\ref{specialcontiguity2}) as a system of
linear equations in the unknown coefficients of $p(t,x)\in M(t)[x]$,
and the unknown $\hat{R}(t)$. 
Suppose $p(t,x)$ and $\hat{R}(t)$ are a non-trivial solution of 
(\ref{specialcontiguity2}). 
Let $u$ be the degree of $\hat{R}(t)$ in $t$
We take $u=-\infty$ 
if $\hat{R}$ is identically zero. 
Since $q(t,x)$ has degree one in $t$, the degree in $t$ of the
second term in the right hand side of 
(\ref{specialcontiguity2}) is strictly larger than the degree in $t$ of the first term in the right hand side.
In particular, if $u < 0$, then taking the limit of  (\ref{specialcontiguity2}) as $t \rightarrow \infty$ we obtain that
$$
n_\gamma(x)=
x(1-x)(1-z_0x)\left(\frac{l}{x}-\frac{(m-l)}{1-x}
+\frac{kz_0}{1-z_0x}\right)
\lim_{t \rightarrow \infty} t\,p(t,x).
$$
The factor  before $\lim_{t\to\infty}$ on the right hand side
is the non-trivial polynomial $q_0(x)$ of degree $\le2$. If it has a zero at $x=0$, then $l=0$.
But that contradicts $x$ dividing $n_\gamma(x)$. Similarly $1$ and $1/z_0$ cannot be zeros
of $q_0(x)$. Suppose that $q_0(x)$ has degree $<2$. That would mean $k-m=0$ and the
degrees of numerator and denominator of $g_\gamma(z_0,x)$ would be the same.
Then, $\deg_x(p(t,x))\le\deg_x(n_\gamma(x))-2$ implies that
$\deg_x(q_0(x)p(t,x))<\deg_x(n_\gamma(x))$,
which again gives a contradiction. We conclude that $u\ge0$.
 
Assume now that $u > 0$ and define $\hat{R}_0=\lim_{t\to\infty}t^{-u}\hat{R}(t)$.
Let us multiply (\ref{specialcontiguity2}) by $t^{-u}$ and then take the
limit as $t\to\infty$.
We obtain that
$$
-\hat{R}_0\, d_\gamma(x)=q_0(x)\Big(\lim_{t\rightarrow \infty}t^{1-u}p(t,x)\Big).
$$
Just as in the previous case, we arrive at a contradiction.

Hence, we conclude that $u = 0$.
Notice that, since the right hand side converges to a non-trivial polynomial,
the degree of $p$ in $t$ has to be $-1$. In particular, we have completed the
proof of the following statement.

\begin{lemma}
\label{lemma:DegreesInt}
If $\beta+t\gamma$ is non-resonant and $p(t,x)$ and $\hat R(t)$ is
a solution to (\ref{specialcontiguity2}), then $\hat R(t)$
has degree $0$ in $t$, and $p(t,x)$ has degree $-1$ in $t$.
\end{lemma}

Let us now summarize our conclusion. We take the point of view that
if $\deg_x(q_0)<2$, then we say that $q_0$ has a zero at $x=\infty$.
In particular, if $q_0$ is constant we say that $q_0$ has a double zero
at $\infty$.

\begin{theorem}\label{belyi}
Let $\beta+t\gamma,z_0$ with $\gamma=(k,l,m)$ be a non-resonant admissible
quadruple. Let $x_1,x_2$ be the zeros of $q_0(x)$,
as defined in (\ref{def:q0}). Then, $z_0$ has the
property that $g_\gamma(z_0,x_1)=g_\gamma(z_0,x_2)$ if $x_1\ne x_2$ 
and $g_\gamma'(x_1)=0$ if $x_1=x_2$.
Moreover, the limit $\hat{R}_0:=\lim_{t\to\infty}\hat R(t)$
is non-zero and given by $g_\gamma(z_0,x_1)$. Consequently, the factor $R_0$
in Theorem \ref{main} is given by $\frac{m^m}{l^l(m-l)^{m-l}}g_\gamma(z_0,x_1)$. 
\end{theorem}

\begin{remark}
When none of $k,l,m-k,m-l$ is zero, the zeros $x_1,x_2$ are distinct from
$0,1,\infty$. The condition $g_\gamma(z_0,x_1)=g_\gamma(z_0,x_2)$ is simply
the requirement that $g_\gamma(z_0,x)$ is a Belyi map. By that we mean a rational
function such that the set of images of its ramification points consists of
at most three points in $\bbbp^1$. 

When one of $k,l,m-k,m-l$ is zero, $g_\gamma$
is automatically a Belyi map, but the condition $g_\gamma(z_0,x_1)=g_\gamma(z_0,x_2)$
still gives a finite number of possibilities for $z_0$.
\end{remark}

\begin{example}
Let us consider the example $\gamma=(2,2,1)$. Then, $g_\gamma(z_0,x)$ is a Belyi map
if and only if $z_0=-\nicefrac18$. There is one non-trivial ramification point of $g_\gamma(-\nicefrac18,x)$,
which is $x=4$. Then, 
$g_\gamma(-\nicefrac18,4)=64/27$, as desired.
\end{example}

\section{Kummer's list}
Let $x \mapsto g(x)$ be a fractional linear transformation in $x$ that
permutes the points $0,1,\infty$. Then, the substitution  $x \mapsto g(x)$ in $K(a,b,c,z,x)dx$,
yields as a result another Euler kernel. For example,
\[
\renewcommand\arraystretch{1.2}
\begin{array}{lll}
g_1(x)=1/x&\quad \mbox{gives}\quad& z^{-a}K(a,a+1-c,a+1-b,1/z,x)dx\\
g_2(x)=1-x&\quad \mbox{gives}\quad& (1-z)^{-a}K(a,c-b,c,z/(z-1),x)dx\\
g_3(x)=x/(x-1)&\quad \mbox{gives}\quad& K(a,b,a+b+1-c,1-z,x)dx\\
g_4(x)=1-1/x&\quad \mbox{gives}\quad& z^{-a} K(a,a+1-c,a+b+1-c,1-1/z,x)dx\\
g_5(x)=1/(1-x)&\quad \mbox{gives}\quad& (1-z)^{-a} K(a,c-b,a+1-b,1/(1-z),x)dx.
\end{array}
\]
We can also consider linear fractional transformations in $x$ that permute the four points
$0,1,\infty,1/z$. These permutations are products of 2-cycles. Up to a constant factor,
\[
\renewcommand\arraystretch{1.2}
\begin{array}{lll}
g_6(x)=1/zx&\quad \mbox{gives}\quad& z^{1-c}K(b+1-c,a+1-c,2-c,z,x)dx\\
g_7(x)=(x-1/z)/(x-1)&\quad \mbox{gives}\quad& z^{1-c}(1-z)^{c-a-b}K(1-b,1-a,2-c,z,x)dx\\
g_8(x)=(1-x)/(1-zx)&\quad \mbox{gives}\quad& (1-z)^{c-a-b}K(c-a,c-b,c,z,x)dx.
\end{array}
\]
Together with the additional substitutions given by $g_i\circ g_j$ for $i=1,\ldots,5$ and $j=6,7,8$ we
get $24$ forms of the shape $\lambda(z)K(a',b',c',h(z),x)dx$. 

Consider the example given by $g_1(x) = 1/x$, which 
changed $K(a,b,c,z,x)dx$ into $z^{-a}K(a,a+1-c,a+1-b,1/z,x)dx$. 
The application of ${\cal L}$ to this form vanishes in $H^1_{\rm twist}$.
As application of the Pochhammer contour yields 
$$z^{-a}F(a,a+1-c,a+1-b\,|\,1/z),$$
the latter is also a solution to the hypergeometric equation. In this way the $24$ forms
obtained from the $24$ rational linear transformations are related to the $24$ Kummer solutions;
the entire list can be seen in Table~\ref{table:Kummer}.

\begin{table}[t]
\begin{center}
\begin{tabular}{cccccc}
$\lambda(z)$ & $h(z)$ & $a'$ & $b'$ & $c'$ & permutation\\
\hline
$1$ & $z$ & $a$ & $b$ & $c$ & $(1)$\\
$(1-z)^{c-a-b}$ & $z$ & $c-a$ & $c-b$ & $c$ & $(13)(24)$\\
$z^{1-c}$ & $z$ & $b-c+1$ & $a-c+1$ & $2-c$ & $(14)(23)$\\
$z^{1-c}(1-z)^{c-a-b}$ & $z$ & $1-b$ & $1-a$ & $2-c$ & $(12)(34)$\\
$z^{-a}$ & $\nicefrac{1}{z}$ & $a$ & $a-c+1$ & $a-b+1$ & $(23)$\\
$z^{-b}$ & $\nicefrac{1}{z}$ & $b-c+1$ & $b$ & $b-a+1$ & $(14)$\\
$z^{b-c}(1-z)^{c-a-b}$ & $\nicefrac{1}{z}$ & $1-b$ & $c-b$ & $a-b+1$ & $(1342)$\\
$z^{a-c}(1-z)^{c-a-b}$ & $\nicefrac{1}{z}$ & $c-a$ & $1-a$ & $b-a+1$ & $(1243)$\\
$1$ & $1-z$ & $a$ & $b$ & $a+b-c+1$ & $(34)$\\
$z^{1-c}(1-z)^{c-a-b}$ & $1-z$ & $1-b$ & $1-a$ & $c-a-b+1$ & $(12)$\\
$z^{1-c}$ & $1-z$ & $b-c+1$ & $a-c+1$ & $a+b-c+1$ & $(1324)$\\
$(1-z)^{c-a-b}$ & $1-z$ & $c-a$ & $c-b$ & $c-a-b+1$ & $(1423)$\\
$(1-z)^{-a}$ & $\nicefrac{z}{z-1}$ & $a$ & $c-b$ & $c$ & $(24)$\\
$(1-z)^{-b}$ & $\nicefrac{z}{z-1}$ & $c-a$ & $b$ & $c$ & $(13)$\\
$z^{1-c}(1-z)^{c-a-1}$ & $\nicefrac{z}{z-1}$ & $1-b$ & $a-c+1$ & $2-c$ & $(1432)$\\
$z^{1-c}(1-z)^{c-b-1}$ & $\nicefrac{z}{z-1}$ & $b-c+1$ & $1-a$ & $2-c$ & $(1234)$\\
$z^{-a}$ & $1-\nicefrac{1}{z}$ & $a$ & $a-c+1$ & $a+b-c+1$ & $(243)$\\
$z^{-b}$ & $1-\nicefrac{1}{z}$ & $b-c+1$ & $b$ & $a+b-c+1$ & $(134)$\\
$z^{a-c}(1-z)^{c-a-b}$ & $1-\nicefrac{1}{z}$ & $c-a$ & $1-a$ & $c-a-b+1$ & $(123)$\\
$z^{b-c}(1-z)^{c-a-b}$ & $1-\nicefrac{1}{z}$ & $1-b$ & $c-b$ & $c-a-b+1$ & $(142)$\\
$z^{1-c}(1-z)^{c-a-1}$ & $\nicefrac{1}{1-z}$ & $1-b$ & $a-c+1$ & $a-b+1$ & $(132)$\\
$(1-z)^{-a}$ & $\nicefrac{1}{1-z}$ & $a$ & $c-b$ & $a-b+1$ & $(234)$\\
$z^{1-c}(1-z)^{c-b-1}$ & $\nicefrac{1}{1-z}$ & $b-c+1$ & $1-a$ & $b-a+1$ & $(124)$\\
$(1-z)^{-b}$ & $\nicefrac{1}{1-z}$ & $c-a$ & $b$ & $b-a+1$ & $(143)$
\end{tabular}
\caption{Kummer's 24 transformations}
\label{table:Kummer}
\end{center}
\end{table}

The last column in Table~\ref{table:Kummer} consists of permutations of $S_4$ in cycle notation. They have
the following meaning. To every triple $a,b,c$ we form the 4-vector
\begin{equation}\label{4vector}
(a-\half,-b+\half,c-a-\half,b-c+\half).
\end{equation}
It turns out that the coordinates of these 4-vectors are permutations of each other
and that every permuation occurs precisely once.

As an example, the fifth entry in Table~\ref{table:Kummer} corresponds to our
example $g_1(x)=1/x$. Notice that
$$(a'-\half,-b'+\half,c'-a'-\half,b'-c'+\half)=(a-\half,c-a-\half,-b+\half,b-c+\half).$$
The latter 4-tuple is the same as (\ref{4vector}) except that the second and third entry
are exchanged. This explains the permutation $(23)$ in the last column of 
Table~\ref{table:Kummer}. 

In \cite[Prop 2.3]{ebisu} we find that to every admissible quadruple there correspond 23 other
admissible quadruples, but with different shift vectors $\gamma=(k,l,m)$. From the above
considerations we can see that they arise from the 24 transforms of the Euler kernel.

Consider the equality (\ref{functionalDR}) which abbreviates as
$$K(\beta+(t+1)\gamma,z_0,x)dx\is\hat{R}(t)K(\beta+t\gamma,z_0,x)dx.$$

Apply any one of the 24 permutation actions of Table~\ref{table:Kummer} to this equality. For
example the permutation $(34)$. We then get the equality
$$K(\beta'+(t+1)\gamma',1-z,x)dx=\hat R(t)K(\beta'+t\gamma',1-z,x)dx,$$
where $\beta'=(a,b,a+b+1-c)$ and $\gamma'=(k,l,k+l-m)$.
Then, $(\beta',z_0)$ is an admissible quadruple with respect to the shift vector $\gamma'$.
Application of the permutation $(1234)$ gives us

\[
K\Big(\beta''+(t+1)\gamma'', \frac{z_0}{z_0-1},x\Big)dx =z_0^{m}(1-z_0)^{l-m}\,\hat R(t)\,K\Big(\beta''+t\gamma'',\frac{z_0}{z_0-1},x\Big)dx,
\]
where $\beta''=(b-c+1,1-a,2-c)$ and $\gamma''=(l-m,-k,-m)$.
Hence we find that $(\beta'',z_0)$ is an admissible quadruple with respect to $\gamma''$.
Since $(34)$ and $(1234)$ generate $S_4$ we find that the transformations $k,l,m\to k,l,k+l-m$
and $k,l,m\to l-m,-k,-m$ generate a group of order $24$ transformations which give us the
$24$ integer triples that occur in Ebisu's \cite[Prop 2.3]{ebisu}.
We can find a fundamental domain for these 24 transformations by requiring that
$k\ge m-k\ge l-m\ge -l$, hence $m\le 2k,2l$ and $k+l\le 2m$. In particular this implies
that $k,l,m\ge0$. Note that this choice differs from Ebisu's normalization $0\le k+l-m\le
l-k\le m$, see \cite[(1.14)]{ebisu}.

\section{Sample $\Gamma$-evaluations}
In this section we collect some examples of $\Gamma$-evaluations related to
non-resonant admissable quadruples. For the resonant cases we refer to Section
\ref{resonant}. Notice that even if one of $k,l,m-k,m-l$ is zero,
one may still have a non-resonant quadruple. Below we find several such examples.
Each entry is preceded by the corresponding shift vector $k,l,m$. In some cases
it may happen that the $c$-parameter tends to a negative integer when the $a$ or $b$
parameter does (see the final remarks in the introduction). In that case we also mention
what the polynomial interpretation gives as value.

\subsection*{The shift $\gamma = (1,3,2)$}
$$F(t,3t-1,2t\,|\,e^{\pi i/3})=-\frac{\sqrt{3}}{2}e^{\pi i(\nicefrac t2+\nicefrac5/6)}\Big(\frac4{\sqrt{27}}\Big)^t
\frac{\Gamma(t+\nicefrac12)\Gamma(\nicefrac13)}{\Gamma(t+\nicefrac13)\Gamma(\nicefrac12)}.
$$
This is the example that was mentioned on page \pageref{anexample}. 
If $t\in\bbbz_{\le 0}$ and the left hand side is considered as finite sum, then the
constant $-\sqrt{3}e^{\pi 5i/3}/2$ must be dropped.

\subsection*{The shift $\gamma = (2,4,4)$}
$$F(2t,4t-\nicefrac12,4t\,|\,-2+\sqrt{8})=\frac{1}
{\sqrt{2}}(1+\sqrt{2})^{4t}
\frac{\Gamma(t+\nicefrac14)\Gamma(t+\nicefrac34)\Gamma(\nicefrac38)\Gamma(\nicefrac58)}
{\Gamma(t+\nicefrac38)\Gamma(t+\nicefrac58)\Gamma(\nicefrac14)\Gamma(\nicefrac34)}.$$
When $2t\in\bbbz_{\le0}$ and the left hand is considered as polynomial, the 
right hand side must be multiplied by $(-1)^{2t}\sqrt{2}$.
\subsection*{The shift $\gamma =(-1,2,1)$}
\begin{align*}
F(-t,2t+1,t+\nicefrac43\,|\,\nicefrac19)
&=\left(\frac34\right)^t\frac{\Gamma(\nicefrac76)\Gamma(t+\nicefrac43)}
{\Gamma(\nicefrac43)\Gamma(t+\nicefrac76)}.\\
F(-t,2t+2,t+\nicefrac53\,|\,\nicefrac19)& =\left(\frac34\right)^t\frac{\Gamma(\nicefrac32)
\Gamma(t+\nicefrac53)}{\Gamma(\nicefrac53)\Gamma(t+\nicefrac32)}.
\end{align*}

\subsection*{The shift $\gamma =(-2,4,2)$}
Let $z_0=(3+2\sqrt{3})/9$ and $z_1=(3-2\sqrt{3})/9$
\begin{align*}
F(-2t,4t+1,2t+\nicefrac43\,|\,z_0) &=
\left(\frac{-27z_1}{16}\right)^t\frac{\cos(\pi(t-\nicefrac1{12}))}{\cos(\pi/12)}
\frac{\Gamma(t+\nicefrac23)\Gamma(t+\nicefrac76)\Gamma(\nicefrac34)\Gamma(\nicefrac{13}{12})}
{\Gamma(t+\nicefrac34)\Gamma(t+\nicefrac{13}{12})\Gamma(\nicefrac23)\Gamma(\nicefrac76)}.\\
F(-2t,4t+1,2t+\nicefrac43\,|\,z_1) &=
\left(\frac{27z_0}{16}\right)^t
\frac{\Gamma(t+\nicefrac23)\Gamma(t+\nicefrac76)\Gamma(\nicefrac34)\Gamma(\nicefrac{13}{12})}
{\Gamma(t+\nicefrac34)\Gamma(t+\nicefrac{13}{12})\Gamma(\nicefrac23)\Gamma(\nicefrac76)}.\\
F(-2t,4t+2,2t+\nicefrac53\,|\,z_0) &=
\left(\frac{-27z_1}{16}\right)^t\frac{\cos(\pi(t+\nicefrac1{12}))}{\cos(\pi/12)}
\frac{\Gamma(t+\nicefrac43)\Gamma(t+\nicefrac56)\Gamma(\nicefrac54)\Gamma(\nicefrac{11}{12})}
{\Gamma(t+\nicefrac54)\Gamma(t+\nicefrac{11}{12})\Gamma(\nicefrac43)\Gamma(\nicefrac56)}.\\
F(-2t,4t+2,2t+\nicefrac53\,|\,z_1) &=
\left(\frac{27z_0}{16}\right)^t
\frac{\Gamma(t+\nicefrac43)\Gamma(t+\nicefrac56)\Gamma(\nicefrac54)\Gamma(\nicefrac{11}{12})}
{\Gamma(t+\nicefrac54)\Gamma(t+\nicefrac{11}{12})\Gamma(\nicefrac43)\Gamma(\nicefrac56)}.
\end{align*}

\subsection*{The shift $\gamma =(-1,-1,1)$}
$$F(-t,-t+\nicefrac13,t+\nicefrac43\,|\,-\nicefrac18)=\left(\frac{27}{32}\right)^t
\frac{\Gamma(t+\nicefrac43)\Gamma(\nicefrac76)}
{\Gamma(t+\nicefrac76)\Gamma(\nicefrac43)}.
$$

\subsection*{The shift $\gamma =(-2,-2,2)$}
Let $z_0=(3\sqrt{3}-5)/4$
$$
F(-2t,-2t+\nicefrac13,2t+\nicefrac43\,|\,z_0)=
\left(\frac{81\sqrt{3}}{128}\right)^t
\frac{\Gamma(t+\nicefrac23)\Gamma(t+\nicefrac76)\Gamma(\nicefrac34)\Gamma(\nicefrac{13}{12})}
{\Gamma(t+\nicefrac34)\Gamma(t+\nicefrac{13}{12})\Gamma(\nicefrac23)\Gamma(\nicefrac76)}.
$$

\subsection*{The shift $\gamma =(0,1,3)$}

$$F(\nicefrac12,t,3t-1\,|\,\nicefrac34)=\frac{1}{3}\frac{\Gamma(t+\nicefrac13)\Gamma(t-\nicefrac13)\Gamma(\nicefrac16)\Gamma(-\nicefrac16)}
{\Gamma(t+\nicefrac16)\Gamma(t-\nicefrac16)\Gamma(\nicefrac13)\Gamma(-\nicefrac13)}.
$$
When $t\in\bbbz_{\le0}$ and the left hand side is considered polynomial, the factor
$1/3$ should be dropped.

\subsection*{The shift $\gamma =(1,3,1)$}

$$F(t,3t-\nicefrac32,t+\nicefrac12\,|\,4)=\frac{6e^{2\pi it}\cos(\pi t)}{27^t}
\frac{\Gamma(t+\nicefrac12)\Gamma(t-\nicefrac12)}
{\Gamma(t-\nicefrac16)\Gamma(t+\nicefrac16)}.
$$

\subsection*{The shift $\gamma =(-1,3,2)$}
\begin{align*}
F(-t,3t+1,2t+\nicefrac32\,|\,\nicefrac14) &=\left(\frac{16}{27}\right)^t
\frac{\Gamma(t+\nicefrac54)\Gamma(t+\nicefrac34)\Gamma(\nicefrac76)\Gamma(\nicefrac23)}
{\Gamma(t+\nicefrac76)\Gamma(t+\nicefrac23)\Gamma(\nicefrac54)\Gamma(\nicefrac34)}.\\
F(-t,3t+2,2t+\nicefrac94\,|\,-\nicefrac18) &=\left(\frac{32}{27}\right)^t
\frac{\Gamma(t+\nicefrac{13}8)\Gamma(t+\nicefrac98)\Gamma(\nicefrac43)\Gamma(\nicefrac{17}{12})}
{\Gamma(t+\nicefrac43)\Gamma(t+\nicefrac{17}{12})\Gamma(\nicefrac{13}8)\Gamma(\nicefrac98)}.
\end{align*}

\subsection*{The shift $\gamma =(3,3,4)$}
$$F(3t, 3t + \nicefrac12, 4t + \nicefrac23\,|\, \nicefrac89)=108^t
\frac{\Gamma(t+\nicefrac{11}{12})\Gamma(t+\nicefrac5{12})\Gamma(\nicefrac12)\Gamma(\nicefrac56)}
{\Gamma(t+\nicefrac12)\Gamma(t+\nicefrac56)\Gamma(\nicefrac{11}{12})\Gamma(\nicefrac5{12})}
$$
When $t\in -1/6+\bbbz_{\le0}$ and the left hand side is considered polynomial one
must multiply the result by $2$.
When $t\in -2/3+\bbbz_{\le0}$ and the left hand side is considered polynomial one
must multiply the result by $-2$.

$$F(3t, 3t + \nicefrac14, 4t + \nicefrac13\,|\, \nicefrac89)=108^t
\frac{\Gamma(t+\nicefrac{7}{12})\Gamma(t+\nicefrac56)\Gamma(\nicefrac34)\Gamma(\nicefrac23)}
{\Gamma(t+\nicefrac34)\Gamma(t+\nicefrac23)\Gamma(\nicefrac{7}{12})\Gamma(\nicefrac56)}
$$
When $t\in \{-1/12, -1/3\}+\bbbz_{\le0}$ and the left hand side is considered polynomial one
must multiply the right hand side by $108^{\nicefrac{1}{12}}\frac{\Gamma(\nicefrac56)
\Gamma(\nicefrac{7}{12})^2}{\Gamma(\nicefrac12)\Gamma(\nicefrac34)^2}$.

$$F(3t, 3t - \nicefrac12, 4t\,|\, \nicefrac43)=\frac{1}{4}(1-\sqrt{-3})e^{2\pi it}16^t
\frac{\Gamma(t+\nicefrac14)\Gamma(t+\nicefrac34)\Gamma(\nicefrac16)\Gamma(\nicefrac23)}
{\Gamma(t+\nicefrac16)\Gamma(t+\nicefrac23)\Gamma(\nicefrac14)\Gamma(\nicefrac34)}
$$
When $t\in\bbbz_{\le0}$ and the left hand side considered polynomial we must drop the factor $(1-\sqrt{-3})/4$.

\subsection*{The shift $\gamma =(2,1,0)$}
\begin{align*}
F(2t, t + \nicefrac16, \nicefrac23\,|\, -8)&=\frac{2}{\sqrt{3}\cdot 27^t}\sin(\pi(t+\nicefrac13))\\
F(2t, t+\nicefrac13,\nicefrac43\,|\,-8)&=\frac{2\cos(\pi(t+\nicefrac13))}{27^t}\frac{\Gamma(t-\nicefrac16)\Gamma(\nicefrac12)}{\Gamma(t+\nicefrac12)\Gamma(-\nicefrac16)}
\end{align*}

\subsection*{The shift $\gamma =(3,1,0)$}
\begin{align*}
F(3t,t+\nicefrac16,\nicefrac12\,|\,9) &=\frac{1}{2\cdot 64^t}\left(1+e^{2\pi i(t+\nicefrac16)}
-e^{4\pi i(t+\nicefrac16)}\right)\\
F(3t,t+\nicefrac12,\nicefrac32\,|\,9) & =\frac{-1}{6\sqrt{3}\cdot64^t}
(1-e^{2\pi it}+e^{4\pi it})\frac{\Gamma(t-\nicefrac16)\Gamma(t+\nicefrac16)}
{\Gamma(t+\nicefrac13)\Gamma(t+\nicefrac23)}
\end{align*}
The first line is proven in Corollary \ref{example3}. Unfortunately 
the second line cannot be proven in this manner because we do not
having enough special values of $t$ with an elementary
evaluation. The result is a conjecture which was found experimentally.
Furthermore we found
\begin{align*}
F(3t,t+\nicefrac16,\nicefrac12\,|\,-3) & =\frac{1}{16^t}\cos(\pi t)\frac{\Gamma(t+\nicefrac12)\Gamma(\nicefrac13)}
{\Gamma(t+\nicefrac13)\Gamma(\nicefrac12)}\\
F(3t,t+\nicefrac12,\nicefrac32\,|\,-3) & =\frac{2}{16^t}\cos(\pi (t + 1/3))\frac{\Gamma(t-\nicefrac16)\Gamma(\nicefrac23)}
{\Gamma(t+\nicefrac23)\Gamma(-\nicefrac16)}
\end{align*}

\subsection*{The shift $\gamma=(1,1,0)$}
Strictly speaking there is no admissable quadruple with respect to $(1,1,0)$. However,
we do like to recall the following classical identity
$$F(t,t+\half,\half|z^2)=\frac12\left((1+z)^{-2t}+(1-z)^{-2t}\right).$$

\end{document}